\documentclass[12pt]{amsart}
\font\emailfont=cmtt10

\usepackage{amsmath,amsthm,amsfonts,amscd,flafter,epsf}
\usepackage{graphics}
\usepackage{epsfig}
\usepackage{psfrag}

\newcommand\commentable[1]{#1}

\newtheorem{theorem}{Theorem}[section]
\newtheorem{prop}[theorem]{Proposition}
\newtheorem{cor}[theorem]{Corollary}
\newtheorem{conj}[theorem]{Conjecture}
\newtheorem{lemma}[theorem]{Lemma}
\newtheorem{claim}[theorem]{Claim}

\newtheorem{defn}[theorem]{Definition}

\newtheorem{remark}[theorem]{Remark}

\def\endproof{\relax\ifmmode\expandafter\endproofmath\else
  \unskip\nobreak\hfil\penalty50\hskip.75em\hbox{}\nobreak\hfil\bull
  {\parfillskip=0pt \finalhyphendemerits=0 \bigbreak}\fi}
\def\endproofmath$${\eqno\bull$$\bigbreak}
\def\bull{\vbox{\hrule\hbox{\vrule\kern3pt\vbox{\kern6pt}\kern3pt\vrule}\hrule}}

\newcommand{\Q}{\mathbb{Q}}

\newcommand{\C}{\mathbb{C}}

\newcommand{\Z}{\mathbb{Z}}
\newcommand{\D}{\mathbb{D}}

\newcommand{\cm}{\cdot}

\newcommand{\F}{\mathbb F}

\newcommand\RelSpinC{\underline{\mathrm{Spin}}^c}
\newcommand\RelSpinCYK{\RelSpinC(Y,K)}
\newcommand\relspinc{\underline{\spinc}}

\newcommand\Filt{\mathcal F}

\newcommand\x{\mathbf x}

\newcommand\ModFlow{\mathcal M}

\newcommand\ModSphere{\ModFlow\left({\mathbb S}\longrightarrow 
\Sym^{g-1}(\Sigma_{1})\times \Sym^2(\Sigma_{2})\right)}
\newcommand\ModSpheres\ModSphere

\newcommand\CFa{\widehat{CF}}

\newcommand\CFinf{CF^\infty}

\newcommand\HFa{\widehat{HF}}

\newcommand\UnparModSp{\widehat \ModSp}
\newcommand\UnparModFlow\UnparModSp
\newcommand\Mod\ModSp

\newcommand{\spinc}{\mathfrak s}

\newcommand\ModMaps{\mathcal M}
\newcommand\ModSp\ModMaps

\newcommand\Ta{{\mathbb T}_{\alpha}}

\newcommand\Tb{{\mathbb T}_{\beta}}

\newcommand\spincrel\relspinc

\newcommand\CFK{CFK}
\newcommand\HFK{HFK}

\newcommand\CFKa{\widehat\CFK}

\newcommand\CFKinf{\CFK^{\infty}}

\newcommand\HFKa{\widehat\HFK}

\newcommand\Dual{\mathcal D}
\newcommand\Duality\Dual

\newcommand\ons{Ozsv{\'a}th and Szab{\'o}}
\newcommand\os{Ozsv{\'a}th-Szab{\'o}}

\newcommand\WD{$D_+(K,p)$}
\newcommand\WDnm{D_+(K,p)}

\newcommand\HFDnm{\HFKa(D_+(K,p),1)}

\newcommand\HFMnm{\underset{\{\relspinc\in \RelSpinCY\}}\bigoplus \!\!\!\!  \!\!\!\! \!\!\!\! \HFKa(S^3_p(K),K',\relspinc)}

\newcommand\RelSpinCY{\RelSpinC(S^3_p(K),K')}

\newcommand\Surg{$S^3_p(K)$}
\newcommand\Surgnm{S^3_p(K)}

\newcommand\kind{$(L(p,q),K')$}
\newcommand\kindnm{L(p,q),K'}

\newcommand\Cij[2]{C\{i{#1},j{#2}\}}
\newcommand\Ci[1]{C\{i{#1}\}}
\newcommand\SpinC{\mathrm{Spin}^c}

\newcommand\filtY[1]{ {\Filt}(K,{#1})}
\newcommand\Sym{\mathrm{Sym}}

\commentable{

\title[{On Floer homology and the Berge conjecture}]{On Floer homology and the Berge conjecture on knots admitting lens space surgeries}

}

\author[Matthew Hedden]{Matthew Hedden}
\address{Department of
Mathematics, Massachusetts Institute of Technology \newline
\indent{\emailfont{mhedden@math.mit.edu}}}

\begin{document}

\begin{abstract}
We complete the first step in a two-part program proposed by Baker, Grigsby, and the author to prove that Berge's construction of knots in the three-sphere which admit lens space surgeries is complete.  The first step, which we prove here, is to show that a knot in a lens space with a three-sphere surgery has simple (in the sense of rank) knot Floer homology.  The second (conjectured) step involves showing that, for a fixed lens space, the only knots with simple Floer homology belong to a simple finite family.  Using results of Baker, we provide evidence for the conjectural part of the program by showing that it holds for a certain family of knots. Coupled with work of Ni, these knots provide the first infinite family of non-trivial knots which are characterized by their knot Floer homology.  As another application, we provide a Floer homology proof of a theorem of Berge.


\end{abstract}

\maketitle
\section{Introduction}
\label{sec:intro}
On which knots in the three-sphere can one perform Dehn surgery and obtain a lens space?  This question has received considerable attention in recent years \cite{Berge,Bleiler,AbsGrad,Lens,Goda1,Ras2,Baker1,Baker2} and much progress has been made towards a general method of enumeration of such knots.   Indeed, there is a conjecture that a  construction due to Berge \cite{Berge} which produces knots in $S^3$ with lens space surgeries is complete (in the sense that any knot admitting a lens space surgery comes from this construction).  The purpose of this paper is discuss a strategy  by which the knot Floer homology theory of \ons \ \cite{Knots} and Rasmussen \cite{Ras1} could prove this conjecture, and to make partial progress towards implementing this strategy.  We will also try to provide evidence supporting the validity of our strategy.

\subsection{Background on the Berge Conjecture}

Before stating our results, we take some time to review the conjecture.  We begin by recalling Berge's construction.

\vspace{4mm}
\noindent{\bf Construction B:} {\em Let $(H_1,H_2,\Sigma)$ be the standard genus two Heegaard splitting of  $S^3$. Here $H_1,H_2$ are genus two handlebodies, joined along their common boundary $\partial H_1=\partial H_2=\Sigma$, a genus two surface.  Let $K\hookrightarrow \Sigma$ be a knot embedded in the Heegaard surface in such a way that $(\iota_k)_*:\pi_1(K) \hookrightarrow \pi_1(H_k)\cong F_2$ represents a generator of the fundamental group of each handlebody, where $$\iota_k: \Sigma \hookrightarrow H_k, k=1,2$$
\noindent are the inclusion maps, and $F_2$ is the free group on two generators. }

\bigskip
Following Berge, we call the knots in the above construction {\em double primitive}.  Performing Dehn surgery on $K$ with framing given by $\Sigma$ can be thought of as attaching a pair of three-dimensional two-handles, $A_1,A_2$, to $H_1$ and $H_2$, respectively, along $K$.  The double primitive condition ensures that the resulting manifolds $V_i=H_i \cup A_i$ have fundamental group $\Z$. The loop theorem then implies each $V_i$ is a solid torus, and hence the manifold obtained by the surgery is a lens space.  We have the following conjecture, which is frequently referred to as the Berge conjecture.
\bigskip

\noindent {\bf Berge Conjecture $1$:}{\em\ \  If $(S^3,K)$ is a knot on which Dehn surgery yields a lens space, then $(S^3,K)$ is double primitive.}

\bigskip
Before proceeding, we make a few observations regarding the surgery slopes in the above conjecture.  Note first that while the conjecture makes no reference to the slope of the surgery, the slope given by Construction B is clearly integral; it is specified by the framing given by $\Sigma$.  However, it follows from the Cyclic Surgery Theorem \cite{Culler} that unless $(S^3,K)$ is a torus knot, any Dehn surgery on $K$ yielding a lens space must be integral.   Since surgeries on torus knots yielding lens spaces are well-understood \cite{Moser} we will henceforth focus attention on integral surgeries unless otherwise specified.

Evocative as the Berge Conjecture may be, for the purpose of explicitly enumerating knots with lens space surgeries it is useful to view the problem from the perspective of the lens space.  To do this, we first observe that a knot $(S^3,K)$ on which Dehn surgery yields the lens space $L(p,q)$, naturally induces a knot \kind.  This knot is the core of the solid torus glued to $S^3-K$ in the surgery. Note that \kind \ is not null-homologous, as it generates the first homology $H_1(L(p,q);\Z)\cong \Z/p\Z$.  Now it is easy to see that \kind \ admits a surgery yielding the three-sphere: simply remove $K'$ and undo the original surgery.  Conversely, if surgery on \kind \ yields $S^3$, there is an induced knot $(S^3,K)$ on which surgery yields $L(p,q)$. Thus the two perspectives are equivalent.

When studying knots in lens spaces admitting $S^3$ surgeries, a natural class of knots arises:
\begin{defn} $(${\bf One-bridge}$)$ A knot $(L(p,q),K')$ is called {\em one-bridge with respect to the standard genus one Heegaard splitting of $L(p,q)$} if  $K'$ is isotopic to a knot which intersects each solid torus of the Heegaard splitting in a single unknotted arc. 
\end{defn} We will hereafter drop the Heegaard splitting from the terminology and simply say that \kind \ is one-bridge. Such knots become relevant in light of Lemma $1$ of \cite{Berge}, which shows that if $(S^3,K)$ is double primitive, then the induced knot \kind \ is one-bridge.    {\em A priori} working with the class of one-bridge knots does not simplify matters much.  Indeed, there are clearly infinitely many one-bridge knots in $L(p,q)$; in particular, it contains torus knots - those knots which can be isotoped to lie in the Heegaard torus - as a proper subset (see \cite{Choi} for a classification scheme). However, amongst the one-bridge knots in $L(p,q)$ is a particularly simple finite subfamily, which we call {\em simple} (or {\em grid-number one}) knots.  To describe these knots, let $(V_\alpha,V_\beta,T^2)$ be the standard genus one Heegaard splitting of $L(p,q)$, and let $D_\alpha$ and $D_\beta$ be the meridian discs of the two solid tori, $V_\alpha,V_\beta$.  Assume that $\partial D_\alpha$,  $\partial D_\beta$ have minimal intersection number i.e. $\partial D_\alpha \cap \partial D_\beta = \{p \ \mathrm{pts}\}$,  see Figure \ref{fig:grid1}.  We then have 
\begin{defn}$(${\bf Simple knot}$)$\label{defn:grid1} A one-bridge knot $(L(p,q),K')$ is {\em simple} if either it bounds a disk, or is the union of two properly embedded arcs, $t_\alpha,t_\beta$, in $D_\alpha$ and $D_\beta$, respectively. See Figure \ref{fig:grid1}. \end{defn}
	Note that there are $p$ simple knots in $L(p,q)$ - there is a unique simple knot in each homology class.  For the reader familiar with \os \ Floer homology, simple knots are precisely those knots which can be realized  by placing two basepoints, $z$ and $w$, on a minimally intersecting Heegaard diagram for $L(p,q)$.  For $S^3$ there is a unique simple knot - the unknot - and it is the connection between knots in lens spaces and grid diagrams that motivates the alternate terminology grid-number one; simple knots are those knots in lens spaces possessing a grid-diagram of grid-number one (see \cite{MOS} and \cite{BGH}). Simple knots and one-bridge knots are important for studying lens space surgeries due to the following theorem of Berge.

\begin{figure}
\begin{center}
\psfrag{a}{$\partial D_\alpha$}
\psfrag{b}{$\partial D_\beta$}
\psfrag{ta}{$t_\alpha$}
\psfrag{tb}{$t_\beta$}
\psfrag{z}{$z$}
\psfrag{w}{$w$}

\caption{\label{fig:grid1}
Depiction of a simple knot $K'$ in $L(7,3)$. Shown is the standard genus one Heegaard diagram of $L(7,3)$ with minimal intersection number $\partial D_\alpha \cap \partial D_\beta$. On it we have depicted a simple knot, $K'$, composed of two arcs $t_\alpha$ and $t_\beta$.  Each arc can be isotoped along $D_\alpha$ (resp. $D_\beta$) to a proper subinterval of $\partial D_\alpha$ (resp.$\partial D_\beta$).  Alternatively, $K'$ could be specified by the two basepoints, $z$ and $w$ (see Definition \ref{def:hd} for this correspondence).  \vspace{3mm}  }
\includegraphics{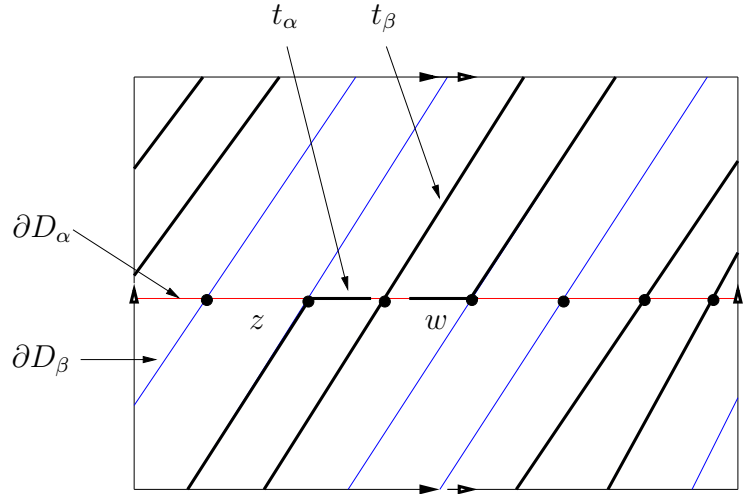}
\end{center}
\end{figure}

\begin{theorem}$($Theorem $2$ of \cite{Berge}$)$ \label{thm:berge} Suppose \kind \ is a one-bridge knot which admits a three-sphere surgery.   Then \kind \ is simple.

\end{theorem}
We present a Floer homology proof of the above theorem in Section \ref{sec:oneone} (though in case the knot induced in the three-sphere by the surgery satisfies $p=2g(K)-1$, with $g(K)$ the Seifert genus, our theorem takes a slightly different form.)  It is straightforward to see that, upon performing the surgery on a simple knot $(L(p,q),K')$, the induced knot $(S^3,K)$ is double primitive.  We are thus led to the useful (equivalent) reformulation of Conjecture $1$, which is stated as a question in \cite{Berge}.

\bigskip

\noindent{\bf Berge Conjecture $2$:} {\em\ \ Suppose \kind \ is a knot which admits a three-sphere surgery.  Then \kind \ is  one-bridge.}

\bigskip
Coupled with Theorem \ref{thm:berge}, an affirmative answer to Conjecture $2$ would allow one to explicitly enumerate all knots in $S^3$ on which surgery could yield a fixed lens space.  To see this, assume surgery on $(S^3,K)$ yields $L(p,q)$.  The above discussion shows that the induced knot \kind\ has an $S^3$ surgery and hence, by Conjecture $2$ and Theorem \ref{thm:berge}, is simple.  Now for each simple knot, $K_i$, there is at most one integral slope surgery producing an integer homology sphere, $M(K_i)$.  Furthermore, the naturally presented Heegaard splitting of $M(K_i)$ is genus two.  One then uses the well-known algorithm to determine if a genus two Heegaard splitting is the three-sphere to determine if $M(K_i)\cong S^3$.  Each $K_i$ for which $M(K_i)\cong S^3$ has an induced (double primitive) knot in $S^3$ on which surgery yields $L(p,q)$.  In this way, a proof of Conjecture $2$ allows a finite enumeration of knots for which surgery yields $L(p,q)$.

\subsection{Statement of results - The role of Floer homology}
Using knot Floer homology, we can hope to prove Conjecture $2$. To do so recall that to any knot $(Y,K)$ in a rational homology sphere (i.e. $H_*(Y;\Q)\cong H_*(S^3;\Q)$), \ons \ associate a collection of bigraded groups (see \cite{RationalSurgeries}),
$$\HFKa(Y,K):=\underset{*\in\Q , \ \xi \in \RelSpinCYK}\bigoplus\HFKa_*(Y,K,\xi).$$
These groups are graded by the Maslov index, which we denote by $*$, and by relative $\SpinC$ structures, $\xi$ \ on $Y-\nu(K)$, the set of which we denote by $\RelSpinCYK$.   The reader unfamiliar with relative $\SpinC$ structures can think of this as a grading by elements of $H_1(Y\!-\!K;\Z)$, since there is an affine isomorphism $\RelSpinCYK\cong H^2(Y,K;\Z)\cong H_1(Y\!-\!K;\Z)$.   This grading should also be viewed as the analog of the Alexander grading on the knot Floer homology of knots in $S^3$.  That is, relative $\SpinC$ structures play the role of the powers of $T$ in the Alexander polynomial of a knot $(S^3,K)$.  

Now the knot Floer homology groups of $(Y,K)$ arise as the associated graded groups of filtrations of the \os \ chain complexes $\CFa(Y,\spinc)$  (here $\spinc$ is a $\SpinC$ structure on $Y$).   Thus there is a spectral sequence which begins with  $\HFKa(Y,K)$ and converges to $\HFa(Y)$, where $\HFa(Y)$ is the direct sum:
$$\HFa(Y):=\underset{\spinc\in \SpinC(Y)}\bigoplus \HFa(Y,\spinc).$$
\noindent It follows immediately that we have the inequality of ranks:
$$\mathrm{rk}\ \HFa(Y) \le \mathrm{rk}\ \HFKa(Y,K)$$ 
We say that a knot has {\em simple} Floer homology if equality holds.  In the case of lens spaces, rk$\ \HFa(L(p,q),\spinc)=1$ for every $\spinc\in \SpinC(L(p,q))$.  Thus $(L(p,q),K')$ has simple Floer homology if  rk$\ \HFKa(L(p,q),K')=p$.

Our first step towards Conjecture $2$ is the following restriction on the knot Floer homology of the knot in $L(p,q)$ induced by the surgery.
\begin{theorem}\label{thm:main}
	Let $(S^3,K)$ be a knot of Seifert genus $g$ and suppose that there exists an integer $p>0$ such that  $p$ surgery on $K$ yields the lens space, $L(p,q)$. Let \kind \ be the knot induced by the surgery. Then
	\begin{enumerate}
		\item	$p\ge 2g-1$,
		\item  If $p\ge 2g$, then \kind \ has simple Floer homology, 
		\item If $p=2g-1$, then  $\mathrm{rk}\ \HFKa(L(p,q),K')=\mathrm{rk}\ \HFa(L(p,q))+2$.
	\end{enumerate}
\end{theorem}
\begin{remark}  By reflecting $K$ if necessary, the assumption that $p$ be positive is non-restrictive. This theorem was recently proved by Rasmussen \cite{Ras3} using a different strategy. \end{remark}
\noindent The fact that $p\ge 2g-1$ is a result first proved by Kronheimer, Mrowka, \ons \ in the context of monopole Floer homology \cite{KMOS}.   Note that when $p=2g-1$,  the induced knot nearly has simple Floer homology - the total rank of the knot Floer homology is only two greater than the Floer homology of the ambient lens space.  We also note that the above theorem is not specific to lens spaces; it also holds if we replace the lens space with an arbitrary L-space. (Recall that an L-space is a rational homology sphere, $Y$, with  rk$\HFa(Y,\spinc)=1$ for each $\spinc\in \SpinC(Y)$.) The Berge conjectures would then follow from

\begin{conj}\label{conj:simp1} (Conjecture $1.5$ of \cite{BGH}) A knot \kind \   with simple Floer homology is simple (in the sense of Definition \ref{defn:grid1}.   
\end{conj}

\begin{conj}\label{conj:simp2} (Conjecture $1.6$ of \cite{BGH})  There are exactly two knots $T_R,T_L \subset L(p,q)$  which satisfy $$ \mathrm{rk}\ \HFKa(L(p,q),T)=\mathrm{rk}\ \HFa(L(p,q))+2.$$\end{conj}

In Section \ref{sec:oneone}, two knots  satisfying $rk(\widehat{HFK}(L(p,q),K)) = p+2$ are specified for each lens space and we show that surgery on them cannot produce $S^3$. Thus a proof of the above conjectures, together with Theorem~\ref{thm:main}, would indeed prove the Berge conjecture. 

 Note that the hypothesis for these conjectures only involves the total rank of the knot Floer homology groups of \kind.  The groups have a rich structure inherited from their bigrading.  It is possible that it would be necessary to exploit this structure.  Thus we are also led to:

\begin{conj}\label{conj:characterize} Let $(L(p,q),G)$ be any simple knot. Suppose that for some knot \kind, we have $$\HFKa_*(L(p,q),K',\relspinc_i)\cong \HFKa_*(L(p,q),G,\relspinc),$$
\noindent for all $*$ and $\relspinc$.  Then $(L(p,q),K')$ is isotopic to $(L(p,q),G)$.  That is, simple knots are characterized by their Floer homology.
\end{conj}

In Section \ref{sec:justify} we provide some justification for the conjectures.  In particular, by using work of Baker \cite{Baker1}, we prove Conjectures \ref{conj:simp1} and \ref{conj:simp2} for knots in $L(p,q)$ satisfying a genus constraint.  To describe this constraint, let us consider only 
those knots \kind \ whose homology class $[K']\in H_1(L(p,q);\Z)$ generates.  For such a knot it makes sense to define the genus of $K'$, denoted $g(K')$, to be the minimum genus of any properly embedded surface-with-boundary 
$$i: (F,\partial F)\hookrightarrow (L(p,q)-\nu K',\partial\nu K')$$
whose homology class is Poincar{\'e} dual to the generator of $H^1(L(p,q)-K';\Z)\cong \Z$.  We then have

\begin{theorem}\label{thm:baker}
	Let \kind \ be any knot whose homology class generates $H_1(L(p,q);\Z)$ and which satisfies $$g(K')\le \frac{p+1}{4}.$$ \noindent Then Conjectures \ref{conj:simp1} and \ref{conj:simp2} hold for \kind.  In particular, if \kind \ has simple Floer homology then \kind \ is simple.  
\end{theorem}

Moreover, by using Baker's result with a result of Ni \cite{YiNi}, it is also possible to prove Conjecture \ref{conj:characterize} for an infinite family of simple knots.  This result seems quite interesting in its own right, as it provides the first infinite family of knots with non-trivial Thurston norm which are characterized by their knot Floer homology (the previous known examples being the figure-eight and trefoil knots).  We briefly discuss this result in Section \ref{sec:justify} and postpone the detailed proof for later \cite{InfiniteFloer}.

\bigskip
\noindent {\bf Outline:}  The next section provides the proof of Theorem \ref{thm:main} which relies heavily on previous work of the author and \ons.  Aided by this theorem, Section \ref{sec:oneone} uses a simple Floer homology argument to prove Berge's Theorem $2$, mentioned above.  While our argument uses the machinery of Floer homology, it avoids the use of the algorithm to detect if a genus two Heegaard splitting is the three-sphere and the Cyclic Surgery Theorem \cite{Culler}. In the final section we discuss evidence for the conjectures and prove Theorem \ref{thm:baker}.  

\bigskip
\noindent {\bf Acknowledgments:} This work has benefited much from conversations with Ken Baker and Eli Grigsby, and the general strategy presented here is part of our joint ongoing work \cite{BGH}.  I also thank Cameron Gordon for generous sharing of his knowledge of Dehn surgery, and Jake Rasmussen for sharing his independent work on Floer homology and the Berge conjecture.


\section{Proof of Main Theorem}
\subsection{Outline}
This section is devoted to a proof of Theorem \ref{thm:main}.  Before beginning, we briefly sketch the idea. Denote the $p$-twisted (positive-clasped) Whitehead double of a knot $K\hookrightarrow S^3$ by $\WDnm$ (this is a specific type of satellite knot, see Figure \ref{fig:Whiteheaddouble}). A formula for the knot Floer homology of $\WDnm$ was exhibited in Theorem $1.2$ of \cite{Doubling}.   This formula was in terms of the filtered chain homotopy  type of the knot Floer homology filtration associated to $K$.  A key step in the proof of the formula was an identification of a particular Floer homology group associated to \WD \ with the direct sum of all the Floer homology groups of the induced knot $(\Surgnm,K')$. Here, \Surg \ denotes the manifold obtained by $p$-surgery on $K$, and $K'$ denotes the knot induced by the surgery i.e. the core of the solid torus glued to $S^3\!-\!K$ in the surgery.  Knowing this identification, we can apply the formula for the Floer homology of the Whitehead double  to calculate the total rank of the Floer homology groups of  $(\Surgnm,K')$.  In the special case that \Surg \ is the lens space $L(p,q)$, \ons \ have an explicit formula (Theorem $1.2$ of \cite{Lens}) which determines the filtered chain homotopy type of the knot Floer filtration of $K$ in terms of the Alexander polynomial of $K$.  Combining these two theorems, Theorem \ref{thm:main} will follow readily.  As with \ons's Theorem, our theorem will handle the more general situation when \Surg \ is an $L$-space, rather than a lens space (recall that an $L$-space is a rational homology sphere  $Y$ for which rk$\ \HFa(Y,\spinc)=1$ for every $\spinc\in\SpinC(Y)$):

\bigskip

\begin{figure}
\begin{center}
\psfrag{N}{t}
\psfrag{N+3}{\small{t+3}}
\psfrag{V}{V}
\psfrag{P}{P}
\psfrag{K}{$\nu K$}
\psfrag{D}{\WD}
\psfrag{1}{1}
\psfrag{f}{f}
\psfrag{+}{+}
\psfrag{=}{=}
\caption{\label{fig:Whiteheaddouble}
The positive $t$-twisted Whitehead double, \WD, of the left-handed trefoil.  Start with a twist knot, $P$, with $t$ full twists embedded in a solid torus, $V$.  The $``+"$ indicates the parity of the clasp of $P$. $f$ identifies $V$ with the neighborhood of $K$, $\nu K$, in such a way that the longitude for $V$ is identified with the Seifert framing of $K$.  The image of $P$ under this identification is \WD. The $3$ extra full twists in the projection of \WD \ shown arise from the writhe of the trefoil, $-3$. }
\includegraphics{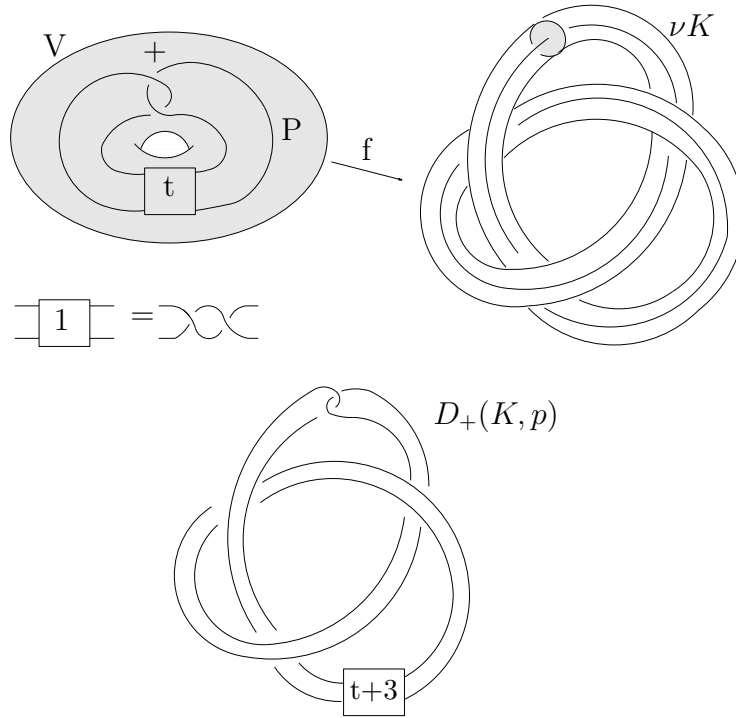}
\end{center}
\end{figure}

\noindent{\bf Theorem \ref{thm:main}} {\em
\	Let $(S^3,K)$ be a knot and suppose that there exists an integer $p>0$ such that  \Surg \ is an $L$-space. Let $(\Surgnm, K')$ be the induced knot.  Then
	\begin{enumerate}
		\item	$p\ge 2g-1$,
		\item  If $p\ge 2g$, then rk$\ \HFKa(\Surgnm,K')= \mathrm{rk}\ \HFa(\Surgnm)=p$,
		\item If $p=2g-1$, then rk$\ \HFKa(\Surgnm,K')=\mathrm{rk}\ \HFa(\Surgnm)+2$.
	\end{enumerate}
	}

	\subsection{Proof of Theorem \ref{thm:main}}
	We begin the details of our proof.    We first note that, by deferring to \cite{KMOS}, we can immediately dispatch with part $(1)$ of the theorem.  Indeed the fact that $p\ge 2g-1$ was proved using monopole Floer homology (with monopole L-space in place of Heegaard Floer L-space) in Corollary $8.5$ of \cite{KMOS} and followed from the algebraic structure of monopole Floer homology together with the existence of a surgery exact sequence.  As the algebraic structure in Heegaard Floer theory is identical and the necessary exact sequence is also in place, the proof in the present context carries over directly. Thus our proof of Theorem \ref{thm:main} will  focuses on parts $(2)$ and $(3)$.  We begin by recalling Theorem $1.2$ of \cite{Doubling}. Throughout the discussion $\F$ will denote the field $\Z/2\Z$, and $\F_{(l)}$ will indicate this same field endowed with Maslov grading $l$.
\begin{theorem}
\label{thm:Meridian}

Let $(S^3,K)$ be a knot and let $(\Surgnm,K')$ be the knot induced in \Surg \ by the core of the surgery torus.  Then
$$\HFDnm \cong \HFMnm,$$

\end{theorem}

\begin{remark} The statement above differs from  that found in \cite{Doubling} in two ways.  First, the statement presented in \cite{Doubling} expresses the knot $(\Surgnm,K')$  as $(\Surgnm,\mu_K)$, where $\mu_K$ is the meridian of the knot $(S^3,K)$ viewed as knot in \Surg.  However, it is straightforward to see that in \Surg,  $\mu_K$ is isotopic to $K'$ - one simply uses the meridian disc of the surgery torus to perform the isotopy.  Second, the right hand side of the congruence is a sum over relative $\SpinC$ structures, instead of a double sum over absolute $\SpinC$ structures on \Surg \ and filtration levels induced by $\mu_K$.  However, as mentioned in the introduction the filtration  of  $\CFa(Y,\spinc)$ induced by $\mu_K$ is by relative $\SpinC$ structures which $\spinc$ extends, and thus the single sum above is equivalent.
\end{remark}

Next, we have Theorem $1.2$ of \cite{Doubling}. To state it, first recall that associated to $S^3$ is the \os \ ``hat'' chain complex, $\CFa(S^3)$, and that the homology of this chain complex is given by $\HFa(S^3)\cong \F_{(0)}$.  Next, recall that to a knot $(S^3,K)$ \ons \ \cite{Knots} associate a filtered version of the chain complex, $\CFa(S^3)$ (see also \cite{Rasmussen}).  That is, we have an increasing sequence of subcomplexes:
$$ 0=\Filt(K,-i) \hookrightarrow \Filt(K,-i+1)\hookrightarrow \ldots \hookrightarrow \Filt(K,n)=\CFa(S^3).$$
\noindent The filtered chain homotopy type of this filtration is an invariant of the pair, $(S^3,K)$.  We denote the quotient complexes $\frac{\Filt(K,j)}{\Filt(K,j-1)}:=\CFKa(K,j)$.  The homology of these quotients, denoted $\HFKa(K,j)$, are the {\em knot Floer homology groups} of $K$.
As in \cite{FourBall}, we define:
$$\tau(K)=\mathrm{min}\{j\in\Z|i_*:H_*(\Filt(K,j))\longrightarrow H_*(\CFa(S^3))\ ~{\text{~is non-trivial}}\}.$$
\noindent This number is the \os \ concordance invariant which, as its name suggests, has been useful in the study of smooth knot concordance, see \cite{FourBall,Livingston}.  In terms of these invariants we have
\begin{theorem}\label{thm:WD}(Theorem $1.2$ of \cite{Doubling}) Let $(S^3,K)$ be a knot with Seifert genus $g(K)=g$.  Then for $p\ge 2\tau(K)$ we have:
	$$\HFKa(\WDnm,1)\cong \F^{p-2g-2}\ \ \bigoplus_{i=-g}^{g}[H(\Filt(K,i))]^2, $$

\noindent Whereas for $p< 2\tau(K)$ the following holds:

	$$\HFKa(\WDnm,1)\cong 
		\F^{-p+ 4\tau(K)-2g-2} \bigoplus_{i=-g}^{g}[H(\Filt(K,i))]^2.$$

	\end{theorem}
	
  Stated above is only the part of Theorem $1.2$ which is relevant for the case at hand, namely the formula for the ``top'' Floer homology group, $\HFKa(\WDnm,1)$. The actual formula is much more general.   Also, we have suppressed here the homological grading of the groups, as we will only be concerned with the ranks.  The astute reader may  question what is meant by a term such as $\F^{-p+4\tau(K)-2g-2}$ this exponent could very well be negative.  By $\F^{-n}$, for instance, we mean the quotient of the remaining group by a subgroup of dimension $n$.

Let us now recall Theorem $1.2$ of \cite{Lens}:

\begin{theorem}
\label{thm:Lens}
	Let $(S^3,K)$ be a knot of Seifert genus $g(K)=g$ and suppose that there exists an integer $p>0$ such that  \Surg \ is an $L$-space. Then there is an
increasing sequence of integers
$$-g=n_{-k}<...<n_k=g$$
with the property that $n_i=-n_{-i}$, and the following significance.
If for $-g\leq i \leq g$ we let
$$\delta_{i}=\left\{\begin{array}{ll}
0 & {\text{if $i=g$}} \\
\delta_{i+1}-2(n_{i+1}-n_{i})+1 &{\text{if $g-i$ is odd}} \\
\delta_{i+1}-1 & {\text{if $g-i>0$ is even,}}
\end{array}\right.$$
then
$\HFKa(K,j)=0$ unless $j=n_i$ for some $i$, in which case
$\HFKa(K,j)\cong \Z$ and it is supported entirely in homological dimension $\delta_i$.
\end{theorem}

We will use Theorem \ref{thm:Lens} together with Theorems \ref{thm:WD} and \ref{thm:Meridian} to deduce Theorem \ref{thm:main}.  

First observe that Theorem \ref{thm:Lens} determines the invariant $\tau(K)$ for knots admitting $L$-space surgeries:
\begin{cor}\label{cor:tau}(Corollary $1.6$ of \cite{Lens}) Let  $(S^3,K)$ be a knot of Seifert genus $g(K)$ and suppose that there exists an integer $p>0$ such that  \Surg \ is an $L$-space. Then $\tau(K)=g(K)$.
\end{cor}
\begin{proof} This follows immediately from the description of the knot Floer homology groups of $K$ given by Theorem \ref{thm:Lens} and the definition of $\tau(K)$.  In particular, the only knot Floer homology group supported in homological grading $0$ is in filtration grading $g(K)$.
\end{proof}

From this corollary, we can insert $g(K)$ in place of $\tau(K)$ in Theorem \ref{thm:WD} and combine the result with Theorem \ref{thm:Meridian} to yield 
\begin{prop}\label{prop:HFM}
 		Let $(S^3,K)$ be a knot with Seifert genus $g(K)=g$, and suppose that there exists an integer $p>0$ such that  \Surg \ is an $L$-space.  Then for $p\ge 2g$ we have:

$$\HFMnm \cong \F^{p-2g-2}\ \ \bigoplus_{i=-g-1}^{g}[H(\Filt(K,i))\oplus H(\Filt(K,-i-1))], $$

\noindent Whereas for $p< 2g$ the following holds:

$$\HFMnm \cong \F^{-p+2g-2}\ \ \bigoplus_{i=-g-1}^{g}[H(\Filt(K,i))\oplus H(\Filt(K,-i-1))], $$	
	\end{prop}

Note that we have chosen to rewrite the direct sum on the far right of the above formulas in a slightly different form.   
To see the equivalence,  note that the adjunction inequality for knot Floer homology (Theorem $5.1$ of \cite{Knots})  implies $H_*(\Filt(K,j))\cong 0$ whenever $j<-g(K)$.  Thus we have:
$$\bigoplus_{i=-g}^{g}[H(\Filt(K,i))]^2 \cong \bigoplus_{i=-g-1}^{g}[H(\Filt(K,i))]^2.$$
\noindent Now it is easy to rewrite the right hand side as it appears in the proposition:
$$\bigoplus_{i=-g-1}^{g}[H(\Filt(K,i))]^2 = \bigoplus_{i=-g-1}^{g}[H(\Filt(K,i))\oplus H(\Filt(K,-i-1))].$$

\noindent The motivation for the manipulation is the following lemma, 
which shows that the rank of  $[H(\Filt(K,i))\oplus H(\Filt(K,-i-1))]$ must be small for knots admitting $L$-space surgeries.

\begin{prop}\label{prop:rank}
	Let $(S^3,K)$ be a knot, and suppose that there exists an integer $p>0$ such that  \Surg \ is an $L$-space. Then $$\mathrm{rk}\ H(\Filt(K,m)) + \mathrm{rk}\ H(\Filt(K,-m-1))=1, \ \ \mathrm{for \ all\ } m.$$
\end{prop}

\begin{proof} The proof relies on a Theorem of \ons \ which relates  the Floer homology of \Surg \ to the Floer homology of the filtration induced on the ``infinity'' chain complex of the three-sphere, $\CFinf(S^3)$, by the knot.  More precisely, there is a refined version of the knot Floer homology filtration described above which associates to a knot $(S^3,K)$ a $\Z\oplus\Z$-filtered chain complex, $\CFKinf(S^3,K)$.   Generators of this chain complex correspond to triples, $[\x,i,j]$, where $\x\in \Ta\cap \Tb$ is an intersection point of two ``lagrangian'' tori in the symmetric product of a Heegaard surface for $S^3$, and $i,j\in \Z$ satisfy a homotopy theoretic constraint:
	$$\langle c_1(\relspinc(\x)),[\Sigma]\rangle+2(i-j)=0$$
	\noindent The above constraint is described by the following: \ons \ assign a relative $\SpinC$ structure, $\relspinc\in \SpinC(S^3\!-\!K)$, to $\x \in \Ta\cap \Tb$. In the present case relative $\SpinC$ structures can be canonically identified with $\SpinC$ structures on $S^3_0(K)$ and, as such, have a well-defined first Chern class, $ c_1(\relspinc(\x))$ which can be evaluated on the generator $[\Sigma]$ of $H_2(S^3_0(K);\Z)\cong\Z$. 
 
	By construction, the generators of $\CFKinf(S^3,K)$ admit a map $$\Filt: \CFKinf(S^3,K)\rightarrow \Z\oplus\Z,$$ determined by $\Filt([x,i,j])=(i,j)$.  If we define a partial ordering on $\Z\oplus\Z$ by $(i,j)\le (i',j')$ if $i\le i'$ and $j\le j'$, then $\Filt$ is a filtration i.e. $\Filt(\partial^\infty [x,i,j] )\le \Filt([x,i,j])$, where $\partial^\infty$ is the differential on $\CFKinf(S^3,K)$ (see \cite{HolDisk,Knots} for a definition of this differential). 
 The rich algebraic structure inherent in a doubly filtered chain complex allows one to examine the homology of many subobjects defined by generators whose filtration indices satisfy specific numerical constraints.  For instance, one can define the chain complex:
 $$\Ci{=0} \subset \CFKinf(S^3,K),$$
\noindent consisting of generators of the form $[\x,0,j]$ for some $j\in\Z$. This set naturally inherits a differential from $\CFKinf(S^3,K)$, since it is a subcomplex of the quotient complex $\frac{\CFKinf}{ \Ci{<0}}$.  We have the chain homotopy equivalence of chain complexes $$\Ci{=0} \simeq \CFa(S^3).$$
\noindent  Thus we recover the \os \ ``hat'' complex of $S^3$.  Furthermore, we have:
$$\filtY{m} \simeq \Cij{=0}{\le m},$$
\noindent and by the filtered chain homotopy equivalence between the filtration associated to a knot $K$ and its reverse, $-K$ (Proposition $3.8$ of \cite{Knots}):
$$\filtY{-m-1} \simeq \Cij{<0}{=m}.$$
\noindent (note that this chain homotopy equivalence does not preserve the Maslov grading, but as the proposition does not reference the Maslov grading we do not belabor this point.) Perhaps the most important aspect of  $\CFKinf(S^3,K)$ chain complex is that it determines the Floer homology groups of manifolds obtained by Dehn surgery on $K$.  Indeed, Theorem $4.1$ of \cite{Knots} states that for an appropriate labeling of $\SpinC$ structures on $S^3_p(K)$ by elements $[m]\in \Z/p\Z \leftrightarrow \SpinC(S^3_p(K))$ we have a chain homotopy equivalence,
$$ H_*(C\{\max(i,j-m)=0\})\cong \HFa(S^3_p(K),\spinc_{[m]}),$$
\noindent for all $p\ge 2g(K)-1$.  Now we have the following short exact sequence of chain complexes:

$$
\begin{CD}
	0 @>>>	\Cij{<0}{=m}
	@>{i}>> C\{\max(i,j-m)=0\} @>{p}>> \Cij{=0}{\le m}	  @>>>0
\end{CD}$$

\noindent Which, under the chain homotopy equivalences mentioned above, leads to a long exact sequence:
$$
\begin{CD}
	\ldots @>>> H(\filtY{-m-1})
	@>{i_*}>> \HFa(S^3_p(K),\spinc_{[m]}) @>{p_*}>> H(\filtY{m})	  @>>>\ldots
\end{CD}$$
\noindent  Under the assumption that $S^3_p(K)$ is an L-space, the middle term has rank one for each $m$.   The proposition will now follow from exactness of the above sequence and the following:
\begin{claim} Let $K\hookrightarrow S^3$ be a knot and suppose that $p>0$ surgery on $K$ yields an L-space.  Then
	$$\mathrm{rk}\ H(\filtY{m})\le 1,$$ 
	\noindent for all $m$.
\end{claim}

We prove the claim with the help of Theorem \ref{thm:Lens} above.   To do this, recall that the knot Floer homology groups can be viewed as a filtered chain complex in their own right, endowed with a differential which strictly lowers the filtration grading.  This is made precise by Lemma $4.5$ of \cite{Ras1}, which we restate:
\begin{lemma} \label{lemma:reduction} (Lemma $4.5$ of \cite{Ras1})
Let  $C$ be a filtered complex with filtration 
$$ C_1 \subset C_2 \subset \ldots \subset C_m,$$
\noindent and let $C^i = C_{i}/C_{i-1}$ be the filtered quotients, so that the homology groups $H_*(C^i)$ are the $E_2$
terms of the spectral sequence associated to the filtration. Then, up to filtered chain homotopy equivalence, there is a unique filtered complex $C'$ with the following properties:
\begin{enumerate} 
	\item  $C'$ is filtered chain homotopy equivalent to $C$.
	\item  $(C')^i \cong H_*(C^i)$
	\item The spectral sequence of the filtration on $C'$ has trivial first differential. Its higher
terms are the same as the higher terms of the spectral sequence of the filtration on
$C$.
\end{enumerate}
\end{lemma}
\noindent In the present situation the lemma allows us to replace the filtered chain complex $(\CFKa(K),\partial)$ by $(\HFKa(K),\partial')$.  Here, $\HFKa(K)$ is meant to indicate the direct sum of the knot Floer homology groups of $K$, $\oplus \HFKa(K,i)$.    In light of this, we have the isomorphisms of Maslov graded groups:
$$ H_*(\Filt(K,m) )\cong H_*(\underset{i\le m}\oplus \HFKa(K,i), \partial'),$$
$$ H_*(\HFKa(K),\partial')\cong H_*(\CFKa(K),\partial)\cong \HFa(S^3)\cong \F_{(0)}.$$
 
 \noindent With this algebraic aside behind us, let us return to the proof of the claim. Assume then, that rk$\ H_*(\Filt(K,m))>1$ for some $m$.  Since $S^3_p(K)$ is an $L$-space, Theorem \ref{thm:Lens} indicates that$$ \mathrm{rk}\ \HFKa(K,j)\le 1 \ \mathrm{for\ all\ } j.$$  \noindent Furthermore, the Maslov gradings of the non-trivial groups are a strictly increasing function of the filtration grading, with  maximum Maslov grading $0$.  This, together with our assumption that rk$\ H_*(\Filt(K,m))>1$ implies  the existence of a subgroup $G\subset H_*(\Filt(K,m))$  satisfying  $$G\cong\F_{(i)}\oplus \F_{(k)}, \ \ i<k\le 0.$$ Using again the fact that the Maslov gradings of the non-trivial Floer homology groups are a strictly increasing function of the filtration grading we have $$\underset{j>m}\oplus \HFKa_{i+1}(K,j)\cong 0.$$
 This in turn implies that the summand $\F_{(i)}\subset G$ survives under the induced differential on $\HFKa(K)$ - there are simply no chains in Maslov grading $i+1$ to map to the generator of $\F_{(i)}$.  Survival of the $\F_{(i)}$ summand, however, contradicts the fact that $\HFa(S^3)\cong \F_{(0)}$.  Thus rk$\ H_*(\Filt(K,m)\le 1$ for all $m$ as claimed.
 \end{proof}
 
 We now complete the proof of Theorem \ref{thm:main}.  Propositions \ref{prop:HFM} and \ref{prop:rank} show that for $p\ge 2g(K)$ we have:
 $$\underset{\relspinc}\Sigma \ \mathrm{rk}\ \HFKa(S^3_p(K),K',\relspinc) = p-2g-2+ \overset{g}{\underset{-g-1}\Sigma}\ 1= p = \mathrm{rk}\ \HFa(S^3_p(K)) , $$	
 \noindent whereas for $p=2g(K)-1$ we have
 $$ \underset{\relspinc}\Sigma \ \mathrm{rk}\ \HFKa(S^3_p(K),K',\relspinc) =	2g-2-p + \overset{g}{\underset{-g-1}\Sigma}\ 1$$ $$ \ \ \ \ = 2g-2-p +2g+2 =4g-p = 2g+1 = \mathrm{rk}\ \HFa(S^3_p(K))+2.$$

\section{A Floer homology proof of Berge's theorem}
\label{sec:oneone}
Recall from the introduction that a that a knot $(L(p,q),K')$ is called {\em zero-bridge with respect to the standard genus one Heegaard splitting of $L(p,q)$} if  $K'$ is isotopic to a knot lying in the torus of the splitting.   A knot $(L(p,q),K')$ is called {\em one-bridge with respect to the standard genus one Heegaard splitting of $L(p,q)$} if  $K'$ is isotopic to a knot which intersects each solid torus of the Heegaard splitting in a single unknotted arc.

In this section we will use Theorem \ref{thm:main} to prove the following:

\begin{theorem} \label{thm:isotope} Suppose \kind \ is a one-bridge and that \kind \ admits a three-sphere surgery. Let $K$ be the knot in $S^3$ induced by the surgery, and let $g$ denote its Seifert genus.  Then either \begin{itemize} \item $p>2g-1$
	, in which case \kind \ is simple	
\item  $p=2g(K)-1$ and \kind \ is one of the two knots, $T_R$, $T_L$ specified by Figure \ref{fig:small} (or their reversals).
	\end{itemize}
\end{theorem}

\begin{figure}
\begin{center}
\psfrag{a}{$\partial D_\alpha$}
\psfrag{b}{$\partial D_\beta$}
\psfrag{ta}{$t_\alpha$}
\psfrag{tb}{$t_\beta$}
\psfrag{TL}{$T_L$}
\psfrag{TR}{$T_R$}
\caption{\label{fig:small}
Doubly-pointed Heegaard diagrams specifying the two knots referenced in Theorem \ref{thm:isotope} and Conjecture \ref{conj:simp2}. Shown are $T_R$ and $T_L$ in the lens space $L(7,3)$.  In $S^3$, $T_R$ and $T_L$ are the right- and left-handed trefoils, respectively.  In general, a Heegaard diagram for $T_R$ and $T_L$ is obtained from the minimal diagram of $L(p,q)$ be a simple isotopy which creates $2$ extra intersection points.  }
\includegraphics{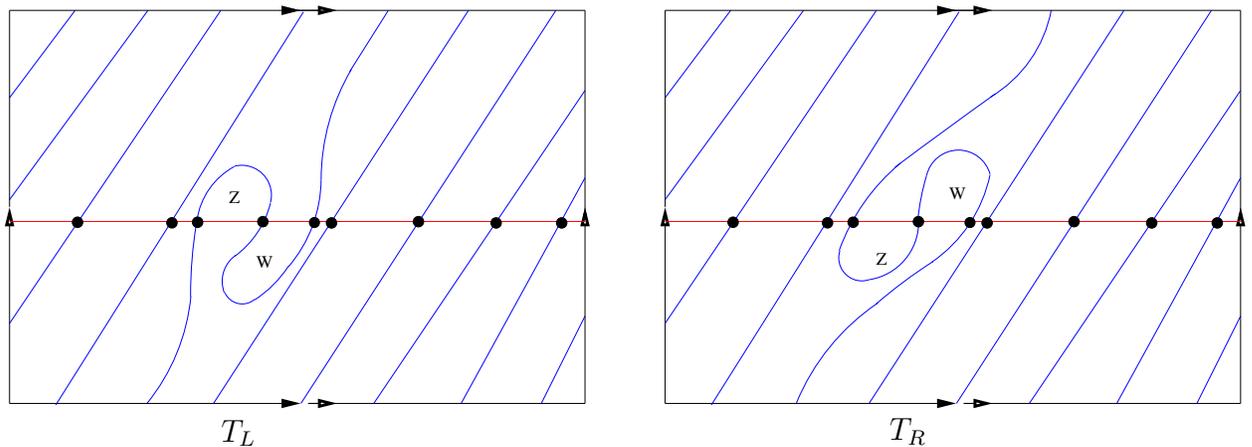}
\end{center}
\end{figure}

We remark that our theorem is more general: it holds for any one-bridge knot with an integral homology sphere L-space surgery.  We further note that from our theorem we immediately recover Berge's theorem (Theorem $2$ of \cite{Berge}) in the case $p>2g-1$.  However, in the case $p=2g-1$ we actually obtain more information; the knot is one of the two knots specified (up to orientation reversal) by Figure \ref{fig:small}.  Of course Berge's theorem tells us that a one-bridge knot in $L(p,q)$ with a three-sphere surgery is simple, independent of $p$.  Thus in the case $p=2g-1$ we obtain
\begin{cor}
	Let \kind \ be a one-bridge knot which admits a three-sphere surgery.  Let $K$ be the knot in $S^3$ induced by the surgery, and $g$ denote its Seifert genus.  Then $p>2g-1$.
\end{cor}
\begin{proof} If $p=2g-1$, then \kind \ is not simple by the above theorem, contradicting Berge's result. 
\end{proof}
	
To prove our theorem, we first note that since \kind \ admits a three-sphere surgery, the induced knot in $S^3$ has a lens space surgery. Thus Theorem \ref{thm:main} applies and we see that either \kind \ has simple Floer homology or $p=2g-1$ and rk$\ \HFKa(L(p,q),K')=\mathrm{rk}\ \HFa(L(p,q))+2=p+2$.   In light of this, the following proposition implies the theorem stated above.

	\begin{prop} \label{prop:oneone}Let \kind \ be a one-bridge knot. 
		\begin{itemize}\item		If \kind \ has simple Floer homology then \kind \ is simple.
			\item If rk$\ \HFKa(L(p,q),K')=p+2$ then  \kind \ is one of the knots specified by Figure \ref{fig:small} (or their reversals). \end{itemize}
	\end{prop} 
	\subsection{Knot Floer homology background for one-bridge knots}
	
Before proving the proposition, we collect some basic facts and definitions surrounding the calculation of knot Floer homology for one-bridge knots in lens spaces.  First recall the definition of a compatible doubly-pointed Heegaard diagram for a knot, $(Y,K)$:   

\begin{defn}
\label{def:hd}
A {\em compatible doubly-pointed Heegaard diagram} for a knot $(Y,K)$ (or simply a Heegaard diagram for $(Y,K)$) is a collection of data 
$$(\Sigma,\{\alpha_1,\ldots,\alpha_g\},\{\beta_1,\ldots,\beta_{g}\},w,z),$$
where  
\begin{itemize}

	\item $\Sigma$ is an oriented surface of genus g, the {\em Heegaard surface},

	\item $\{\alpha_1,\ldots,\alpha_g\}$ are pairwise disjoint, linearly
  independent embedded circles (the {\em $\alpha$ attaching circles}) which specify a handlebody, $U_\alpha$,
  bounded by $\Sigma$, 

	\item $\{\beta_1,\ldots,\beta_{g}\}$ are pairwise disjoint, linearly
  independent embedded circles which specify a handlebody, $U_\beta$,
  bounded by $\Sigma$ such that $U_\alpha\cup_{\Sigma}U_\beta$ is
  diffeomorphic to $Y^3$,

	\item  $K$ is isotopic to the union of two arcs joined along their common endpoints $w$ and $z$.  These arcs, $t_{\alpha}$ and $t_{\beta}$, are 
	properly embedded and unknotted in the $\alpha$ and $\beta$-handlebodies, respectively.  

\end{itemize}
\end{defn}

See Figure \ref{fig:small} for an example.  It was first pointed out in Proposition $2.3$ of \cite{Goda2} that knots which are at most one-bridge with respect to the standard Heegaard splitting of $S^3$  are precisely those knots admitting a genus one Heegaard diagram (note that \cite{Goda2} refers to one-bridge knots as $(1,1)$ knots).  Two points are worth mentioning here.  The first is that Proposition $2.3$ of \cite{Goda2} actually states that there is a genus two Heegaard diagram for a one-bridge knot in $S^3$.  However, it follows from Condition (iii) in the statement of Proposition $2.3$ - namely that the boundary of the meridian disc to $K$ (the first $\beta$ attaching curve) intersects exactly one attaching curve for the $\alpha$ handlebody in exactly one point - that after performing handleslides and isotopies of the attaching curves the Heegaard diagram can be destabilized to a genus one diagram.  Indeed, this observation was exploited in \cite{Goda2} throughout the discussion.  The second observation is that Proposition $2.3$ applies more generally to any one-bridge knot in a lens space.  In particular, knots in $L(p,q)$ that are at most one-bridge admit genus one Heegaard diagrams.

In Section $6$ of \cite{Knots}, \ons \ develop a general technique for calculating the Floer homology of any knot admitting a genus one Heegaard diagram (again, this was only explicitly stated for $Y=S^3$ but holds for any lens space).  Very briefly, recall that knot Floer homology was first defined as the ``Lagrangian'' intersection Floer homology for two totally real submanifolds in the $g$-fold symmetric product of the Heegaard surface (with an appropriate almost complex structure).  These totally real submanifolds are defined by the attaching curves of the Heegaard diagram.  However, in the case at hand - when we are dealing with a genus one Heegaard surface - the construction can be described much more concisely.

\begin{figure}
\begin{center}
\psfrag{C}{$\C$}
\psfrag{ea}{$e_\alpha$}
\psfrag{eb}{$e_\beta$}
\psfrag{i}{$i$}
\psfrag{-i}{$-i$}
\psfrag{x}{$x$}
\psfrag{y}{$y$}
\psfrag{u}{$u$}
\psfrag{x'}{$x'$}
\psfrag{y'}{$y'$}
\psfrag{u'}{$u'$}

\caption{\label{fig:disk}
Depiction of disks in $\pi_2(x,y)$ counted in the incidence number $n(x,y)$.  Shown are two maps of disks, $u$ and $u'$.  Both satisfy the boundary conditions necessary to be in $\pi_2(x,y)$ (resp. $\pi_2(x',y')$), and both are orientation preserving.  Because the image of $u'$ has an obtuse corner, it is not counted in the incidence number $n(x',y')$.  }
\includegraphics{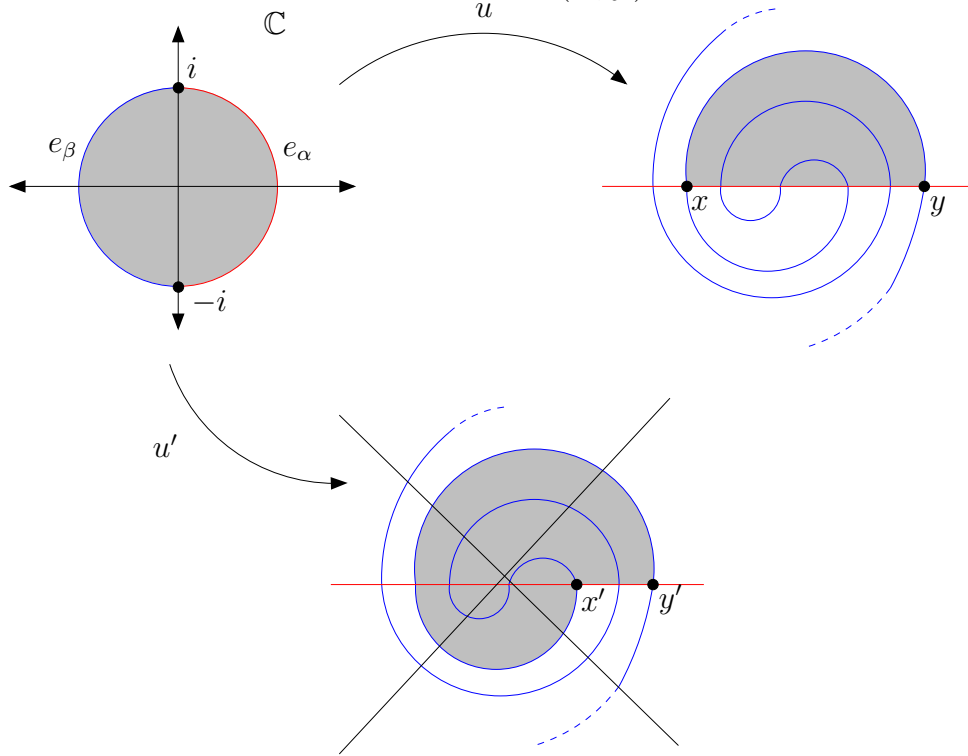}
\end{center}
\end{figure}

Given a  genus one Heegaard diagram for a knot \kind, $$(T^2,\alpha,\beta,w,z),$$ we construct a chain complex $\CFKa(\kindnm)$ as follows.  The generators are intersection points of the attaching curves:
$$\CFKa(\kindnm)=\underset{x\in \alpha\cap \beta}\oplus \F\cm<x>.$$
\noindent Here, for simplicity, we take coefficients in the field with two elements, $\F=\Z/2\Z$.  To define the boundary operator, let us associate an incidence number to $x,y\in \alpha\cap \beta$ as follows.  To begin let $\pi_2(x,y)$ denote the set of homotopy classes of maps of $5$-tuples:
$$u:(\D^2,e_\alpha,e_\beta,-i,i)\rightarrow  (T^2,\alpha,\beta,x,y),$$
\noindent where $\D^2\subset \C$ is the unit disc, $e_{\alpha}$ (resp. $e_{\beta}$) is the part of its boundary with positive (resp. negative) real part, $i=\sqrt{\!-\!1}$, and the right hand side is the Heegaard diagram.  We set 

$$n(x,y)=\left\{\begin{array}{ll} 
	1   &  \mathrm{There\ exists\ an\ orientation}\!-\!\mathrm{preserving \ } u\in \pi_2(x,y)\ \\ 
	&  \mathrm{with\ no\ obtuse\ corners\ and\ } z,w\notin{Im}(u)\\
	0  & \mathrm{otherwise}\\
	
\end{array}\right.$$
	\noindent See Figure \ref{fig:disk} for an explanation of these conditions. In terms of these incidence numbers, the boundary operator can be described by:
	$$\partial x = \underset{y\in\alpha\cap \beta}\Sigma n(x,y)y$$
\noindent The knot Floer homology groups are the homology groups of this chain complex:
$$\HFKa(\kindnm):= H_*(\CFKa(\kindnm),\partial).$$
\noindent For our present purposes, we make no reference to the gradings of this group, but remark that there is a bigrading on the generators of $\CFKa(\kindnm)$ coming from the Maslov index and relative $\SpinC$ structures on $L(p,q)\!-\!K'$.  We will only have need for the total rank of the knot Floer homology groups, summing over both gradings.

\subsection{Proof of Proposition \ref{prop:oneone}}
With the above preliminaries behind us, the Proposition will follow quickly.

\begin{proof} We handle first the case when \kind \ is a one-bridge knot with simple Floer homology.  First observe that simple knots are exactly those knots which can be described by doubly-pointed genus one Heegaard diagrams with minimal intersection number between the $\alpha$ and $\beta$ curves i.e.  \kind \ is simple if and only if it has a genus-one doubly pointed Heegaard diagram with exactly $p$ intersection points.  Now suppose that \kind \ is one-bridge but not simple. Let $(T^2,\alpha,\beta,z,w)$ be a compatible doubly-pointed Heegaard diagram for \kind \ with the fewest number of intersection points $x\in\alpha\cap \beta$.  Since \kind \ is not simple we must have $\underset{\mathrm{geom}}\#|\alpha\cap \beta|>p$.  Thus, if \kind \ has simple Floer homology we have 
	$$\mathrm{rk}\ \CFKa(\kindnm)> \mathrm{rk}\ \HFKa(\kindnm)=p.$$
	In order for the above inequality to hold there must exist a pair of generators, $x,y$, for which $n(x,y)=1$.  This implies the existence of a map from a disc $u$, as above, whose image misses both basepoints $z,w$ defining the Heegaard diagram.  We can use this disc, as in Figure \ref{fig:isotope}, to remove both intersection points (strictly speaking, the isotopy may remove multiple intersection points depending on whether there are other discs present).  In this way we arrive at a Heegaard diagram with strictly fewer intersection points than the one we started with, contradicting our assumption that  $(T^2,\alpha,\beta,z,w)$ was minimal with respect to geometric intersection number.

\begin{figure}
\begin{center}
\caption{\label{fig:isotope} If there exists a non-trivial differential for the chain complex $\CFa(L(p,q,K')$, then the Heegaard diagram for $K$ can be simplified by an isotopy. 
 }
\includegraphics{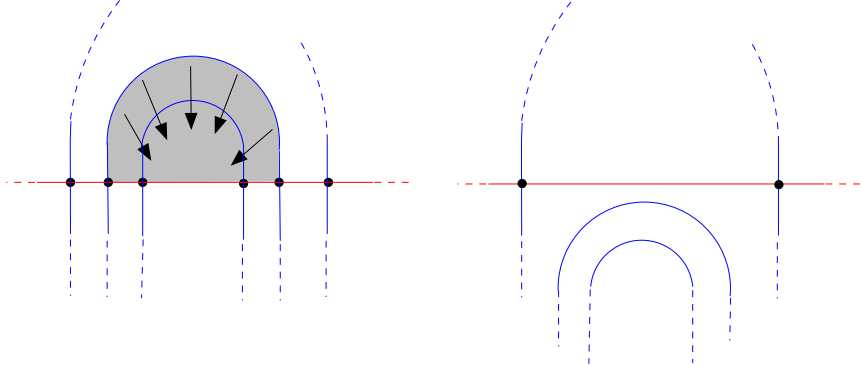}
\end{center}
\end{figure}

	The case when rk$\ \HFKa(L(p,q),K')=p+2$ is only slightly more involved.  As above, let $(T^2,\alpha,\beta,z,w)$ be a compatible doubly-pointed Heegaard diagram for \kind \ with the fewest number of intersection points $x\in\alpha\cap \beta$.  By the isotopy argument described above, we can assume this number to be $p+2$.  That is
	\begin{equation} \label{eq:knot} \mathrm{rk}\ \CFKa(\kindnm)= \mathrm{rk}\ \HFKa(\kindnm)=p+2.\end{equation}
		To show that \kind \ is the knot of Figure \ref{fig:small}, we note that there is a refined incidence number:
		$$n_z(x,y)=\left\{\begin{array}{ll} 
	1   &  \mathrm{There\ exists\ an\ orientation}\!-\!\mathrm{preserving \ } u\in \pi_2(x,y)\ \\ 
	&  \mathrm{with\ no\ obtuse\ corners\ and\ } w\notin{Im}(u)\\
	0  & \mathrm{otherwise}\\
	
\end{array}\right.$$
\noindent and that the operator defined by 
	$$\partial_z x = \underset{y\in\alpha\cap \beta}\Sigma n_z(x,y)y$$
satisfies $\partial_z\circ \partial_z=0$.  Further, the homology of the resulting chain complex satisfies
\begin{equation}\label{eq:Y} \mathrm{rk}\ H_*(\CFKa(\kindnm),\partial_z)=p.\end{equation}
Similar remarks hold if we switch the roles of $z$ and $w$.  That is, there exists an analogous boundary operator $\partial_w$ on $\CFKa(\kindnm)$ whose homology is also of rank $p$.  
Now Equations \eqref{eq:knot},\eqref{eq:Y} imply the existence of intersection points $x,y\in \alpha\cap\beta$, and a disk $u\in \pi_2(x,y)$ which satisfies:
\begin{itemize}\item $u$ is orientation preserving
	\item Im$(u)$ has no obtuse corners
	\item $\underset{alg}\# z\cap \mathrm{Im}(u)=\underset{geom}\# z\cap \mathrm{Im}(u)\ \ge 1$
\end{itemize}
Where the last intersection is the algebraic or geometric intersection number of the image of $u$ with the submanifold $z\hookrightarrow T^2$ (the fact that the two numbers are equal follows from the fact that $u$ is orientation preserving). We claim that in fact, $\# z\cap \mathrm{Im}(u)= 1$.  Indeed, if this were not the case, we could lift the Heegaard diagram and Im$(u)$ to the universal cover of $T^2$ to show that there also exist intersection points $x',y'$ in $\alpha\cap\beta$ and a disk $u\in\pi_2(x',y')$ satisfying $\# z\cap \mathrm{Im}(u)= 1$ (and also the other two conditions itemized above).  This, in turn, implies that $\underset{\mathrm{geom}}\#|\alpha\cap \beta|>p+2$, contradicting our assumption.

Using $\partial_w$, a similar discussion shows the existence of intersection points $r,s\in \alpha\cap\beta$, and a disk $v\in \pi_2(r,s)$ satisfying $\# w\cap \mathrm{Im}(v)= 1$

The proposition now follows from the claim that, of the four intersection points $x,y,r,s$, two must be equal:  for if two of the intersection points are equal, then the existence of the disks $u$ and $v$ force the Heegaard diagram for \kind \ to take the form of Figure \ref{fig:small}.  However, the claim follows from the observation that if none of $x,y,r,s$ are equal,  $\underset{\mathrm{geom}}\#|\alpha\cap \beta|>p+2$.
\end{proof}


\section{Intuition for Conjectures \ref{conj:simp1}-\ref{conj:simp2} and proof of Theorem \ref{thm:baker}}
\label{sec:justify}
The purpose of this Section is to provide some justification for why the conjectures cited in the introduction would be true.  Loosely speaking, Conjecture \ref{conj:simp1} and \ref{conj:simp2} say that simple Floer homology implies simple knot.  Conjecture \ref{conj:characterize} says that simple knots are characterized by their Floer homology. 
 
We first point out that all three conjectures hold in the case where $L(p,q)=S^3$.  This is made precise by the theorems of \ons \ and Ghiggini:
\begin{theorem}\label{thm:S3} (Theorem $1.2$ of \cite{GenusBounds})
Suppose $K \subset S^3$ satisfies $rk(\widehat{HFK}(S^3,K)) = 1$. Then $K$ is the unknot (the only simple knot in $S^3$). 
\end{theorem}
\begin{theorem}\label{thm:trefoil}(Corollary $1.5$ of \cite{Ghiggini})
Suppose $K \subset S^3$ satisfies  $rk(\widehat{HFK}(S^3,K)) = 3$.  Then $K$ is the right- or left-handed trefoil.
\end{theorem}

For knots in a general lens space, perhaps the most compelling quantitative evidence at the moment is Theorem \ref{thm:baker}, which says that Conjectures \ref{conj:simp1} and \ref{conj:simp2} hold for knots whose complements have somewhat simple topology.  

\bigskip
\noindent {\bf Proof of Theorem \ref{thm:baker}.}  Suppose that \kind \ is a knot whose homology class generates $H_1(L(p,q);\Z)$ and for which  \begin{equation} \label{eq:genus} g(K')\le \frac{p+1}{4}.\end{equation} Theorem $1.1$ of \cite{Baker1} shows that \kind \ is one-bridge.  Now Proposition \ref{prop:oneone} applies and shows that if \kind \ has simple Floer homology then  \kind \ is simple.  Strictly speaking, we should also address the case when rk$\ \HFKa(L(p,q),K')=p+2$. However, in this case any knot satisfying the genus constraint would be one-bridge and, by Proposition \ref{prop:oneone}, would be the knot of Figure \ref{fig:small}.  However, it can be shown (using, for example, results of Ni \cite{YiNi}) that this knot does not possess surfaces in its complement which satisfy the genus constraint.   $\square$

\bigskip
For knots in $L(p,q)$ which satisfy Equation \eqref{eq:genus}, one can also prove Conjecture \ref{conj:characterize}.  To describe the method by which this is done, assume that we are given a simple knot $(L(p,q),K')$ satisfying \eqref{eq:genus}.  The calculation of the bigraded Floer homology groups of any simple knot is straightforward. Now a theorem of Ni \cite{YiNi} shows that the breadth of the homology support in the Alexander grading determines the genus of any knot $(L(p,q),J)$.  Thus, if another knot $(L(p,q),J)$ had the same Floer homology as a simple knot satisfying \eqref{eq:genus}, $J$ would necessarily satisfy \eqref{eq:genus} as well.  By Baker's theorem, this would imply that $J$ is one-bridge and Proposition \ref{prop:oneone} would imply that $J$ is simple.  Then one can check that there is a unique simple knot in $L(p,q)$ with the Floer homology of $(L(p,q),K')$.  We postpone the details of this argument for a later time, but suffice it to say that it is straightforward to find infinite families of simple knots (in different lens spaces) of arbitrarily large genus which are characterized by their knot Floer homology.

\bigskip
In another direction, the recent connection between knot Floer homology and grid diagrams \cite{MOS} also provides compelling evidence for Conjecture $3$.  In \cite{MOS} it was shown that the Floer homology of knots in $S^3$ can be combinatorially computed from a certain (grid) diagram associated to a particular (grid) projection of $K$.  The extension of this combinatorial formula  to knots in lens spaces is discussed  in \cite{BGH}.  There, a similar formula to that for knots in $S^3$ is presented which computes the knot Floer homology of an arbitrary knot $(L(p,q),K')$.  This formula, too, is in terms of a grid diagram for \kind.  To date, there are no combinatorial proofs of Theorems \ref{thm:S3} and \ref{thm:trefoil}.  The existing proofs rely on connections between \os \ theory and symplectic geometry.  However, it seems reasonable to expect that if either or both of these theorems could be understood combinatorially,  then the proofs could be adapted to the more general setting of simple knots in lens spaces.  Indeed, the combinatorics of grid diagrams for knots in lens spaces is completely analogous to that of knots in the three-sphere \cite{BGH}.

\commentable{
\bibliographystyle{plain}
\bibliography{biblio}
}

\end{document}